\documentclass[10pt, reqno]{amsart}
\numberwithin{equation}{section}
\usepackage{amssymb, color,xcolor}
\usepackage{hyperref}
\usepackage{cite}
\newcommand{\qtq}[1]{\quad\text{#1}\quad}

\newcommand{\R}{\mathbb{R}}
\newcommand{\C}{\mathbb{C}}

\newcommand{\eps}{\varepsilon}

\newtheorem{theorem}{Theorem}[section]

\newtheorem{lemma}[theorem]{Lemma}

\newtheorem{corollary}[theorem]{Corollary}

\newtheorem{proposition}[theorem]{Proposition}

\theoremstyle{definition}
\newtheorem{definition}[theorem]{Definition}
\newtheorem{remark}[theorem]{Remark}

\theoremstyle{remark}

\usepackage{mathtools}
\mathtoolsset{showonlyrefs}

\begin{document}

\title[Dispersion-managed NLS]{Averaging for the dispersion-managed NLS}

\author{Luccas Campos}
\address{IMECC, State University of Campinas (UNICAMP), Campinas, SP, Brazil}
\email{luccas@ime.unicamp.br}
\author{Jason Murphy} 
\address{Missouri University of Science \& Technology}
\email{jason.murphy@mst.edu}
\author{Tim Van Hoose}
\address{Missouri University of Science \& Technology}
\email{trvkdb@mst.edu}

\begin{abstract} We establish global-in-time averaging for the $L^2$-critical dispersion-managed nonlinear Schr\"odinger equation in the fast dispersion management regime.  In particular, in the case of nonzero average dispersion, we establish averaging with any subcritical data, while in the case of a strictly positive dispersion map, we obtain averaging for data in $L^2$. 
\end{abstract}

\maketitle

\section{Introduction}

We consider dispersion-managed nonlinear Schr\"odinger equations (DMNLS) of the form 
\begin{equation}\label{DMNLS}
i\partial_t u + \gamma(t)\Delta u = \mu |u|^p u,\quad (t,x)\in\R\times\R^d,
\end{equation}
where the dispersion map $\gamma:\R\to\R$ is a $1$-periodic function.  This model arises in nonlinear fiber optics, particularly when $\gamma$ is taken to be piecewise constant (see, for example, \cite{Agrawal, Kurtzke}).  Equations of the form \eqref{DMNLS} (along with some variants) have been the topic of some recent mathematical interest; we refer the reader to \cite{AntonelliSautSparber, CHL, CL, EHL, Fanelli, GT1, GT2, GMO, HL, HL2, MV, MVH, PZ, ZGJT} for a representative sample of results.  Throughout this note, we will focus on the case of the cubic NLS in two space dimensions, which is of particular interest in the optics setting.  In the case of constant dispersion, this leads to an $L^2$-critical problem (that is, the scaling symmetry that preserves the class of solutions also preserves the $L^2$-norm), although (as we will discuss below) much of our analysis can be adapted to more general choices of $(d,p)$. 

Our interest in this note is the phenomenon of `averaging' in the so-called \emph{fast dispersion management regime}.  In particular, for each $\eps>0$, we consider the equation
\begin{equation}\label{DMNLSeps}
\begin{cases}
i\partial_t u + \gamma(\tfrac{t}{\eps})\Delta u = |u|^2 u, \\
u|_{t=0}=\varphi.
\end{cases}
\end{equation}
Our aim to prove that in the limit $\eps\to 0$, the solutions to \eqref{DMNLSeps} converge in a suitable sense to the solution to the equation
\begin{equation}\label{DMNLS0}
\begin{cases}
i\partial_t u + \langle \gamma\rangle\Delta u = |u|^2 u, \\
u|_{t=0}=\varphi,
\end{cases}
\end{equation}
where $\langle\gamma\rangle$ is the average dispersion, defined by
\[
\langle\gamma\rangle=\int_0^1\gamma(t)\,dt.
\]

A related problem (known as \emph{strong dispersion management}) instead considers maps of the form $\tfrac{1}{\eps}\gamma(\tfrac{t}{\eps})$, where one expects convergence to a non-local effective model (see e.g. \cite{CL, Lushnikov, GT1, EHL, HL, PZ, ZGJT}); this problem is related to the study of so-called \emph{dispersion managed solitons}. See also \cite{dBD, DT} for the case of averaging with random dispersion maps.

In the fast dispersion management regime, the averaging problem has been previously studied in works including \cite{AntonelliSautSparber, BK1, BK2, YKT}. In particular, \cite{AntonelliSautSparber} proved that for a piecewise-constant dispersion map $\gamma$ and $\varphi\in H^2$, one can obtain averaging in the $L_t^\infty H_x^2$-norm on any finite subinterval of the maximal lifespan of the solution to \eqref{DMNLS0}.  In their case, the map $\gamma$ itself must be bounded away from zero, although the case $\langle\gamma\rangle=0$ is permitted.  The high regularity on $\varphi$ is imposed in order to utilize the embedding $H^2\hookrightarrow L^\infty$ in dimensions $d\leq 3$.  

Our main goals in this work are twofold: (i) to reduce the regularity required on the initial condition, and (ii) to obtain global-in-time averaging results, when possible.  Our analysis also permits a slightly more general class of dispersion maps than the piecewise-constant case, although $\gamma$ itself must remain bounded away from zero (and the piecewise-constant case is probably the most important example, anyway).  In particular, we will prove that for suitable dispersion maps with nonzero average dispersion and $\varphi \in H^s$ for some $s>0$, we can obtain averaging in $S:=L_t^\infty L_x^2\cap L_{t,x}^4$ on any interval on which the solution to the averaged equation exists and obeys finite $S$-norms.  As the $2d$ cubic equation is $L^2$-critical, we therefore obtain averaging for any scaling-subcritical data.  In fact, under stronger assumptions on $\gamma$ (namely, strict positivity), we can obtain the same averaging result for merely $L^2$ initial data, yielding a critical result.  

For the case of zero average dispersion, we offer no improvement over the result of \cite{AntonelliSautSparber}.  Indeed, our arguments rely heavily on the use of Strichartz estimates (both for the averaged equation and the $\eps$-dependent equations), which break down in this setting.  In general, it seems difficult to treat the case of zero average dispersion without working with bounded solutions. 

To state our main results precisely, we first introduce the notion of an admissible dispersion map, as defined in \cite{MVH}:

\begin{definition}\label{D:admissible} We call $\gamma:\R\to\R$ \emph{admissible} if:
\begin{itemize}
\item $\gamma(t+1)=\gamma(t)$ for all $t\in\R$.
\item $\gamma$ and $\tfrac{1}{\gamma}$ belong to $L_t^\infty(\R)$.
\item $\gamma$ has at most finitely many discontinuities in $[0,1]$.
\end{itemize}
\end{definition}
In particular, given admissible $\gamma$ with non-zero average dispersion, \cite{MVH} established global-in-time Strichartz estimates for the underlying dispersion-managed Schr\"odinger equation.  Our Theorem~\ref{T}(i) below relies fundamentally on these estimates; in particular, this ultimately explains why we can work with (nearly) critical data. 

Next, we recall the main result of \cite{Dodson}, which states that for $\varphi\in L^2$ and $\langle\gamma\rangle>0$, the solution to \eqref{DMNLS0} with $u|_{t=0}\in L^2(\R^2)$ is global in time, scatters in $L^2$, and satisfies $u\in S(\R),$ where
\begin{equation}\label{def:S}
\|u\|_{S(I)}:=\|u\|_{L_t^\infty L_x^2(I\times\R^2)}+\|u\|_{L_{t,x}^4(I\times\R^2)}. 
\end{equation}

Our main result is as follows: 

\begin{theorem}\label{T} Let $\gamma$ be an admissible dispersion map satisfying either 
\begin{itemize}
\item[(i)] $\int_0^1\gamma(t)\,dt>0$, or
\item[(ii)] $\gamma(t)>0$ for all $t\in\R$.
\end{itemize}

In case (i), let $s>0$ and $\varphi \in H^s(\R^2)$.  In case (ii), let $\varphi\in L^2(\R^2)$.

Let $u:\R\times\R^2\to\C$ be the solution to \eqref{DMNLS0} with $u|_{t=0}=\varphi$. For all $\eps>0$ sufficiently small, the solution $u^\eps$ to \eqref{DMNLSeps} with $u^\eps|_{t=0}=\varphi$ is global in time, with $u^\eps\in S(\R)$ and 
\begin{equation}\label{convergence!}
\lim_{\eps\to 0}\|u^\eps-u\|_{S(\R)}= 0. 
\end{equation}
\end{theorem}

We have arranged matters so that the underlying averaged equation is the defocusing model, in which case all solutions enjoy global space-time bounds.  More generally, the analysis will show that we can obtain averaging on any interval on which the solution to the averaged equation has finite space-time bounds.  Thus, for example, if $\langle\gamma\rangle<0$, we would obtain global-in-time averaging for masses below a suitable ground state threshold.  We have stated Theorem~\ref{T} in this way to emphasize the fact that our result \emph{allows} for global-in-time averaging, when it is in fact possible. 

Similarly, we have chosen to prove our result only for the $L^2$-critical problem.  This makes it clear that in case (i) we are only slightly off from the critical result, while (ii) obtains the critical result in a restricted setting.  More generally, our analysis should carry over to settings in which one can obtain a reasonable well-posedness and stability theory using traditional Strichartz spaces.  Notably, however, this would exclude the important $1d$ cubic model.

In the rest of the introduction, we briefly give the ideas behind the proof of Theorem~\ref{T}.

In case (i), the proof builds upon the strategy of \cite{AntonelliSautSparber}.  In particular, we run a continuity argument based on examining the difference of the Duhamel formulas for the solutions $u^\eps$ and $u$.  As in \cite{AntonelliSautSparber}, the key is to exploit the fact that the underlying linear propagator for \eqref{DMNLSeps} should converge to that of \eqref{DMNLS0}.  In contrast to \cite{AntonelliSautSparber}, we use a quantitative form of this convergence, which costs us slightly in regularity and ultimately explains why Theorem~\ref{T}(i) is an inherently subcritical result.  With this bound in place, we can then interpolate with the Strichartz estimates of \cite{MVH} (which, by a change of variables, can be shown to hold uniformly in $\eps$).  With the full range of Strichartz estimates at our disposal, we can run a fairly standard continuity argument to establish the desired convergence (globally in time).  In fact, in this setting we obtain a quantitative bound of the form
\[
\|u^\eps-u\|_{S(\R)} \lesssim \eps^c
\]
for some small $c=c(s)>0$, where $H^s$ is the regularity of the initial condition. 

In case (ii), we combine ideas from \cite{Fanelli, GMO}, which studied NLS with time-dependent dispersion, and \cite{Ntekoume}, which studied homogenization for NLS with a spatially inhomogeneous nonlinearity.  In particular, in the case of a strictly positive dispersion map, we can utilize a change of variables as in \cite{Fanelli, GMO} to convert \eqref{DMNLSeps} into an equation with constant dispersion, but with a time-dependent factor in the nonlinearity.  We then follow the approach of \cite{Ntekoume} to obtain homogenization for this equation.  The basic idea is to show that the solution to the averaged equation defines an approximate solution to the $\eps$-dependent equation.  That this is possible relies on the fact that the solution itself obeys good space-time bounds, while the difference in the nonlinear factors should vanish as $\eps\to 0$.  To witness this vanishing, however, we first need to perform an integration by parts in time, which introduces a few minor technical complications.  In the end, however, we can obtain the the desired homogenization result.  We then `undo' the change of variables in order to reach the desired averaging result in Theorem~\ref{T}(ii).

The rest of this paper is organized as follows: Section~\ref{S:preliminary} collects some preliminary lemmas.  We then prove Theorem~\ref{T}(i) in Section~\ref{S:i} and Theorem~\ref{T}(ii) in Section~\ref{S:ii}.

\subsection*{Acknowledgements} L. C. was financed by grant $\#$2020/10185-1, S\~ao Paulo Research Foundation (FAPESP). J.M. was supported by a Simons Collaboration Grant.  We are grateful to R. Killip for suggesting the reference \cite{Ntekoume}.

\section{Preliminaries}\label{S:preliminary}  We write $A\lesssim B$ to denote $A\leq CB$ for some $C>0$.  We denote dependence of implicit constants on parameters via subscripts.  For example, $A\lesssim_\gamma B$ means $A\leq CB$ for some $C=C(\gamma)$.

We write $\langle\nabla\rangle^s$ to denote the Fourier multiplier operator $(1-\Delta)^{s/2}$.  We employ the standard Littlewood--Paley frequency projectors, denoted $P_{\leq N}$, $P_N$, $P_{>N}$, and so on.  We will also make use of the standard Bernstein inequalities, e.g.
\begin{align*}
\|P_{>N} f\|_{L^p} &\lesssim N^{-s} \||\nabla|^s f\|_{L^p}, \\
\| |\nabla|^s P_{\leq N}f\|_{L^p} &\lesssim N^s \|f\|_{L^p}, \\
\| P_{\leq N} f\|_{L^q} &\lesssim N^{\frac{d}{r}-\frac{d}{q}}\|P_{\leq N} f\|_{L^r}\quad(r\leq q),
\end{align*}
and so on. 

We will use the notation $\text{\O}(X)$ to denote a finite linear combination of terms that resemble $X$ up to frequency projection and/or complex conjugation.  In particular, we will write
\[
|u|^2 u = |u_{\leq N}|^2 u_{\leq N} + \text{\O}(u_{>N} u^2). 
\]

We will need to make use of the Strichartz estimates for the dispersion-managed equation established in \cite{MVH}.  To state these precisely in two dimensions, recall that we call $(q,r)$ an \emph{admissible pair} if $2<q,r\leq\infty$ and $\tfrac{2}{q}+\tfrac{2}{r}=1$. (In higher dimensions, the endpoint $q=2$ is allowed.) Theorem~2 in \cite{MVH} showed that for admissible dispersion maps $\gamma$ (in the sense of Definition~\ref{D:admissible}) with $\langle \gamma\rangle\neq 0$, we have 
\begin{equation}\label{MVH-Strichartz}
\bigl\| e^{i\Gamma(\cdot,s)\Delta}\|_{L^2\to L_t^q L_x^r} \lesssim_\gamma 1
\end{equation}
uniformly in $s\in\R$, where $\Gamma(t,s):=\int_s^t\gamma(\tau)\,d\tau$.  Combining this with the method of $TT^*$ and the Christ--Kiselev lemma, this yields the full range of Strichartz estimates (minus the $L_t^2$ endpoint, in the higher-dimensional case). 

For $\eps>0$, we now define 
\begin{equation}\label{gamma-eps}
\gamma_\eps(t)=\gamma(\tfrac{t}{\eps}) \qtq{and} \Gamma_{\eps}(t,t_0)=\int_{t_0}^t \gamma_\eps(\tau)\,d\tau.
\end{equation}

We will see below that \eqref{MVH-Strichartz} holds for $\Gamma_\eps$ as well, with implicit constant independent of $\eps$ (cf. \eqref{uniform-Strichartz}).

The following elementary estimate plays a key role in much of what follows.  Indeed, it explains exactly why one should expect convergence to the averaged equation. 

\begin{lemma}\label{L:crucial-estimate} Let $\gamma$ be an admissible dispersion map, with $\langle\gamma\rangle=\int_0^1\gamma(t)\,dt$. For any $t,t_0\in\R$,
\begin{equation}\label{crucial-estimate}
\bigl| \Gamma_\eps(t,t_0)-\langle\gamma\rangle(t-t_0)\bigr| \leq 4\eps\bigl\{\|\gamma\|_{L^\infty}+|\langle\gamma\rangle|\bigr\}.
\end{equation}
\end{lemma}

\begin{proof} We change variables to write
\begin{align*}
\Gamma_\eps(t,t_0) & = \eps\int_{\frac{t_0}{\eps}}^{\frac{t}{\eps}}\gamma(\tau)\,d\tau \\
&= \eps\biggl[\int_{\frac{t_0}{\eps}}^{\lceil \frac{t_0}{\eps}\rceil} \gamma(\tau)\,d\tau + \int_{\lfloor \frac{t}{\eps}\rfloor}^{\frac{t}{\eps}} \gamma(\tau)\,d\tau\biggr] + \eps\int_{\lceil \frac{t_0}{\eps}\rceil}^{\lfloor\frac{t}{\eps}\rfloor}\gamma(\tau)\,d\tau, 
\end{align*}
and we similarly decompose
\[
\langle\gamma\rangle(t-t_0)=\eps\langle\gamma\rangle\bigl[ \tfrac{t}{\eps}-\lfloor\tfrac{t}{\eps}\rfloor - (\tfrac{t_0}{\eps} - \lceil\tfrac{t_0}{\eps}\rceil)\bigr] +\eps\langle\gamma\rangle\bigl(\lfloor \tfrac{t}{\eps}\rfloor - \lceil\tfrac{t_0}{\eps}\rceil\bigr). 
\]

Using the fact that 
\[
\int_{\lceil \frac{t_0}{\eps}\rceil}^{\lfloor\frac{t}{\eps}\rfloor}\gamma(\tau)\,d\tau = \langle\gamma\rangle\bigl(\lfloor \tfrac{t}{\eps}\rfloor - \lceil\tfrac{t_0}{\eps}\rceil\bigr),
\]
we obtain \eqref{crucial-estimate}.
\end{proof}

Finally, we record a stability result for the class of equations of the form
\begin{equation}\label{E:NLSmu}
    i\partial_t u + \Delta u = \mu(t)|u|^2 u,\quad \mu\in L_t^\infty(\R).
\end{equation}

A similar result was proven in \cite{Ntekoume}, with a spatially-dependent coefficient rather than a time-dependent coefficient.  As the proof is essentially identical, we will omit it and instead refer the reader to \cite[Theorem~3.4]{Ntekoume}. 


\begin{proposition}\label{L:stability} Let $\mu \in L^\infty(\R)$ and $I=[0,T]$.  Suppose $\tilde u$ satisfies
\begin{equation}
\begin{cases} 
i\partial_t\tilde{u} + \Delta \tilde{u} = \mu(t)|\tilde u|^2\tilde u + e \\
\tilde u|_{t=0}=\tilde u_0\in L^2.
\end{cases}
\end{equation}

Assume that
\[
\|\tilde{u}\|_{L_t^\infty L_x^2(I \times \R^2)} \leq M,\quad\|u_0 - \tilde{u}_0\|_{L_x^2} \leq M',\qtq{and} 
\|\tilde{u}\|_{L_{t,x}^4(I \times \R^2)} \leq L
\]
for some $M, M', L >0$ and $u_0 \in L^2(\R^2)$. There exists $\eta_1=\eta_1(M,M',L,\|\mu\|_{L^\infty})$ sufficiently small so that if 
\begin{align}
\|e^{it\Delta}(u_0 - \widetilde{u}_0)\|_{L_{t,x}^4(I \times \R^2)} \leq \eta \qtq{and}
\left\| \int_0^t e^{i(t-s)\Delta}e(s) \,ds \right\|_{L_{t,x}^4} \leq \eta,
\end{align}
then there exists a unique solution $u$ to \eqref{E:NLSmu} with $u|_{t=0} = u_0$ obeying 
\begin{align}
\|u - \widetilde{u}\|_{L_{t,x}^4} \lesssim \eta,\quad\|u - \widetilde{u}\|_{S(I)} \lesssim M',\qtq{and}
    \|u\|_{S(I)} &\lesssim 1.
\end{align}
\end{proposition}

This result will be used in Section~\ref{S:ii}, where the coefficients $\mu$ will depend on the parameter $\eps$.  Importantly, however, the $L^\infty$-norm of these coefficients will \emph{not} depend on $\eps$, and hence the parameter $\eta_1$ appearing in the statement of Proposition~\ref{L:stability} will be uniform in $\eps$. 

\section{Proof of Theorem~\ref{T}(i)}\label{S:i}

Throughout this section, we fix an admissible dispersion map $\gamma$ satisfying
\[
\langle\gamma\rangle = \int_0^1 \gamma(t)\,dt>0. 
\]

Given $\eps>0$, we define $\gamma_\eps$ and $\Gamma_\eps$ as in \eqref{gamma-eps}.

\begin{proposition}[Convergence of propagators]\label{P:CoP} There exists $C_\gamma>0$ such that for any $t_0\in\R$ and any $\theta\in[0,1]$, 
\begin{equation}\label{E:CoP}
\|e^{i\Gamma_\eps(\cdot,t_0)\Delta}-e^{i\langle\gamma\rangle(\cdot-t_0)\Delta}\|_{\dot H_x^\theta\to L_t^\infty L_x^2} \leq C_\gamma\eps^{\frac{\theta}{2}}.
\end{equation}
\end{proposition}

\begin{proof} This estimate is essentially contained in \cite{AntonelliSautSparber}.  We fix $t_0,t\in\R$, $\theta\in[0,1]$, and and $\varphi \in \dot H^\theta$. Using \eqref{crucial-estimate}, we estimate
\begin{align*}
\|& (e^{i\Gamma_\eps(t,t_0)\Delta}-e^{i\langle\gamma\rangle(t-t_0)\Delta})\varphi\|_{L^2}^2  \\
& = \int \bigl| (e^{-i[\Gamma_\eps(t,t_0)-\langle\gamma\rangle(t-t_0)]|\xi|^2}-1)\hat\varphi(\xi)\bigr|^2\,d\xi \\
& \leq \int |\Gamma_\eps(t,t_0)-\langle \gamma\rangle(t-t_0)|^{\theta} \ \bigl| |\xi|^{\theta}\hat \varphi(\xi)|^2\,d\xi  \lesssim_\gamma \eps^{\theta}\|\varphi\|_{\dot H^\theta}^2.
\end{align*}
\end{proof}

\begin{corollary}\label{C:CoP} Let $\eps>0$ and $\theta\in[0,1]$.
\begin{itemize}
\item[(i)] For any Schr\"odinger admissible pair $(q,r)$ and any $t_0\in\R$, 
\begin{equation}\label{interp-1}
\| e^{i\Gamma_\eps(\cdot,t_0)\Delta}-e^{i\langle\gamma\rangle(\cdot-t_0)\Delta}\|_{H_x^{\theta}\to L_t^q L_x^r} \lesssim_\gamma \eps^{(1-\frac{2}{q})\frac\theta2}. 
\end{equation}
\item[(ii)] For any admissible pairs $(q,r)$, $(\tilde q,\tilde r)$ we have
\begin{equation}\label{interp-2}
\biggl\| \int_{t_0}^t [e^{i\Gamma_\eps(t,s)\Delta}-e^{i\langle\gamma\rangle(t-s)\Delta}]F(s)\,ds \biggr\|_{L_t^q L_x^r} \lesssim_\gamma \eps^{(2-\frac{2}{q}-\frac{2}{\tilde q})\frac\theta2} \|\langle\nabla\rangle^{2\theta}F\|_{L_t^{\tilde q'}L_x^{\tilde r'}}. 
\end{equation}
\end{itemize}
\end{corollary}

\begin{proof} Fix an admissible pair $(q,r)$, $t_0\in\R$, and $\eps>0$, and define 
\[
u(t,x)=e^{i\Gamma_{\eps}(t,t_0)\Delta}\varphi. 
\]
It follows that
\[
v(t,x):=u(\eps t,\sqrt{\eps}x) \qtq{solves}\begin{cases} i\partial_t v + \gamma(t)\Delta v = 0, \\ v(\tfrac{t_0}{\eps},x) = \varphi(\sqrt{\eps}x), \end{cases} 
\]
i.e.
\[
v(t,x) = e^{i\Gamma_1(t,\eps^{-1}t_0)\Delta}[\varphi(\sqrt{\eps}x )]. 
\]

Thus, by \cite[Theorem~2]{MVH}, we have
\begin{equation}\label{str-with-eps}
\|u(\eps t,\sqrt{\eps} x)\|_{L_t^q L_x^r} \lesssim_\gamma \|\varphi(\sqrt{\eps}x)\|_{L^2}, 
\end{equation}
where the implicit constant is independent of $\eps$. As $(q,r)$ is an admissible pair, a change of variables reduces the estimate \eqref{str-with-eps} to 
\[
\|u(t,x)\|_{L_t^q L_x^r} \lesssim_\gamma \|\varphi\|_{L^2}.
\]

It follows that 
\begin{equation}\label{uniform-Strichartz}
\|e^{i\Gamma_\eps(t,t_0)\Delta}\|_{L^2\to L_t^q L_x^r}\lesssim_\gamma 1 \qtq{uniformly in}\eps>0.
\end{equation}

As we have the same bounds for $e^{i\langle\gamma\rangle(\cdot-t_0)\Delta}$, we can now obtain \eqref{interp-1} by the triangle inequality and interpolation with \eqref{E:CoP}. 

We now deduce (ii) from (i) using the method of ${TT}^*$ and the Christ--Kiselev lemma \cite{CK}.
\end{proof}

\begin{remark} We can obtain these estimates in other dimensions, as well, as long as we avoid the $L_t^2$ endpoint in (ii).
\end{remark}

We are now in a position to prove Theorem~\ref{T}(i).

\begin{proof}[Proof of Theorem~\ref{T}(i)] We let $u$ be the solution to \eqref{DMNLS0} with $u|_{t=0}=\varphi$ and $u^\eps$ the solution to \eqref{DMNLSeps} with $u^\eps|_{t=0}=\varphi$.  

By the result of \cite{Dodson} and persistence of regularity for \eqref{DMNLS0}, we have that $u$ exists globally in time and satisfies
\[
\|\langle\nabla \rangle^s u\|_{S(\R)}\leq M <\infty,
\]
where $S(\cdot)$ is as in \eqref{def:S}.  

By the standard local well-posedness theory for \eqref{DMNLSeps} (see e.g. \cite{AntonelliSautSparber} or \cite{MVH}), we have that $u^\eps$ exists at least locally in time; furthermore, as long as $u^\eps$ remains bounded in the $S$-norm, we can continue the solution. In fact, the time of dependence for $u^\eps$ is independent of $\eps$, which can be deduced from the fact that the implicit constants in the Strichartz estimates for \eqref{DMNLSeps} are independent of $\eps$ (cf. \eqref{uniform-Strichartz}).

Thus, in what follows, it will suffice to assume that the solution $u^\eps$ exists and then establish \emph{a priori} bounds on the difference between $u$ and $u^\eps$.  Without loss of generality, we estimate the difference for times $t\in[0,\infty)$.  In general, the implicit constants below may depend on $\gamma$, but will be independent of $\eps$.  We will write $C_\gamma>0$ to explicitly indicate various implicit constants appearing in the Strichartz estimates below. 

Writing $F(z)=|z|^2 z$, we first use the Duhamel formula to obtain 
\begin{align}
u^\eps(t) - u(t) & = e^{i\Gamma_\eps(t,t_0)\Delta}[u^\eps(t_0)-u(t_0)] \label{dif-data1} \\
& \quad + [e^{i\Gamma_\eps(t,t_0)\Delta}-e^{i\langle\gamma\rangle(t-t_0)\Delta}]u(t_0) \label{dif-data2} \\
&\quad + \int_{t_0}^t e^{i\Gamma_\eps(t,s)\Delta}[F(u^\eps(s))-F(u(s))]\,ds \label{dif-non}\\
& \quad + \int_{t_0}^t [e^{i\Gamma_\eps(t,s)\Delta}-e^{i\langle\gamma\rangle(t-s)\Delta}]F(u(s))\,ds \label{dif-prop}
\end{align}
for any $t,t_0\in\R$. Thus, for any time interval $I\ni t_0$, we can use Corollary~\ref{C:CoP}, Strichartz (see \eqref{uniform-Strichartz}), and the fractional product rule to obtain the \emph{a priori} estimate
\begin{align*}
\|&u^\eps-u\|_{S} \\
& \lesssim \|u^\eps(t_0)-u(t_0)\|_{L^2} + \eps^c \|u\|_{L_t^\infty H_x^{\frac{s}{2}}} + \| F(u^\eps)-F(u)\|_{L_{t,x}^{\frac43}} + \eps^c \|\langle\nabla\rangle^{s}F(u)\|_{L_{t,x}^{\frac43}} \\
& \lesssim \|u^\eps(t_0)-u(t_0)\|_{L^2} + M\eps^c + \bigl\{ \|u^\eps\|_{L_{t,x}^4}^2+\|u\|_{L_{t,x}^4}^2\bigr\}\|u^\eps-u\|_{L_{t,x}^4} + \\
& \quad + \eps^c \| u\|_{L_{t,x}^4}^2 \|\langle \nabla\rangle^{s}u\|_{L_{t,x}^4} \\
& \lesssim \|u^\eps(t_0)-u(t_0)\|_{L^2} + M\eps^c  + \|u\|_{L_{t,x}^4}^2 \|u^\eps-u\|_{L_{t,x}^4} + \|u^\eps-u\|_{L_{t,x}^4}^3 \\
&\quad + M \eps^c \|u\|_{L_{t,x}^4}^2 \\
& \lesssim \|u^\eps(t_0)-u(t_0)\|_{L^2} + M\eps^c  + \|u\|_{L_{t,x}^4}^2 \|u^\eps-u\|_{S} + \|u^\eps-u\|_{S}^3 + M\eps^c \|u\|_{L_{t,x}^4}^2
\end{align*}
for some $c=c(s)>0$, where all space-time norms are over $I\times\R^2$.

We now let $\eta>0$ be a small parameter to be determined below (ultimately depending only on $\gamma$), and split $[0,\infty)$ into $J=J(M,\gamma)$ intervals $I_j=[t_j,t_{j+1}]$ such that
\begin{equation}\label{franged}
\|u\|_{L_{t,x}^4(I_j\times\R^2)}<\eta \qtq{for all}j\geq 0.
\end{equation}

We will prove by induction that for $\eps=\eps(M,\gamma,s)$ sufficiently small, 
\begin{equation}\label{induction}
\|u^\eps - u\|_{S(I_j)} \leq A_j \eps^c \qtq{for all}j\geq 0,
\end{equation}
where 
\[
A_0:=8C_\gamma M\qtq{and} A_{j}:=4C_\gamma[A_{j-1}+2M]\qtq{for} j\geq 1.
\]

Note that as $J=J(M,\gamma)$ is finite, we have that $A_j \lesssim_{\gamma,M} 1$ uniformly in $j$. 

For $j=0$, we note that $u^\eps(t_0)=u(t_0)=\varphi$, so that using \eqref{franged}, the \emph{a priori} estimate on an interval $I=[0,t]\subset I_0$ reduces to
\begin{align*}
\|u^\eps - u\|_{S(I)} \lesssim 2M\eps^c + \eta^2\|u^\eps-u\|_{S(I)} + \|u^\eps-u\|_{S(I)}^3. 
\end{align*}

Choosing $\eta=\eta(\gamma)$ small enough, this implies
\[
\|u^{\eps} - u\|_{S(I)} \leq 4C_\gamma M\eps^c + C_\gamma\| u^\eps-u\|_{S(I)}^3.
\]

As this holds for any $I=[0,t]\subset I_0$, a continuity argument then implies (for $\eps=\eps(\gamma,M,s)$ small enough) 
\[
\|u^{\eps}-u\|_{S(I_0)}\leq 8C_\gamma M \eps^c, 
\]
which is \eqref{induction} for $j=0$.

Now suppose that that \eqref{induction} holds up to level $j-1$.  Then, in particular, we have
\[
\|u^\eps(t_j)-u(t_j)\|_{L^2} \leq A_{j-1} \eps^c.
\]

Applying again the \emph{a priori} estimate above, we find that for any $I=[t_j,t]\subset I_j$, 
\[
\|u^\eps-u\|_{S(I)} \lesssim A_{j-1}\eps^c + 2M\eps^c + \eta^2\|u^\eps-u\|_{S(I)} + \|u^\eps-u\|_{S(I)}^3
\]
which (for $\eta$ small, as before) yields
\[
\|u^\eps-u\|_{S(I)} \leq 2C_\gamma[A_{j-1}+2M]\eps^c + C_\gamma\|u^\eps-u\|_{S(I)}^3 
\]

As long as $\eps=\eps(M,\gamma,s)$ is small enough (depending on $M$, $C_\gamma$, and $\sup_j A_j$), this yields
\[
\|u^{\eps}-u\|_{S(I_j)}\leq 4C_\gamma[A_{j-1}+2M]\eps^c=A_j\eps^c,
\]
which is \eqref{induction} at level $j$.  This completes the proof of \eqref{induction}.

Adding the estimates \eqref{induction} and recalling that $J=J(M,\gamma)$ is finite, we derive
\[
\|u^\eps-u\|_{S([0,\infty))} \lesssim_{\gamma,M} \eps^c,
\]
which yields the result. \end{proof}

\section{Proof of Theorem~\ref{T}(ii)} \label{S:ii}

Throughout this section we assume that $\gamma$ is an admissible dispersion map satisfying $\gamma(t)>0$ for all $t\in\R$.  In fact, by the definition of admissibility, this guarantees that $\gamma$ is bounded away from zero.

We define $\gamma_\eps(t)$ and $\Gamma_\eps(t,t_0)$ as in \eqref{gamma-eps}.  To simplify notation, we will abbreviate $\Gamma_\eps(t,0)$ by $\Gamma_\eps(t)$.  Note that $\gamma_1$ corresponds to the original dispersion map $\gamma$; accordingly, we will denote the function $\Gamma_1$ by $\Gamma$. 

The strategy in this section will be to utilize a change of variables as in \cite{Fanelli, GMO} to convert \eqref{DMNLSeps} into an equation with constant dispersion but time-dependent nonlinearity (see \eqref{NLSeps} below).  We then adapt the arguments of \cite{Ntekoume}, which considered the homogenization problem for $2d$ cubic NLS with spatially dependent nonlinearity, to obtain a homogenization result for \eqref{NLSeps}. Changing back to the original equation, we can complete the proof of Theorem~\ref{T}(ii). 

As $\gamma_\eps$ is bounded away from zero, the function $\Gamma_\eps$ is invertible.  We denote $c_\eps=\Gamma_\eps^{-1}$ and observe  that
\begin{equation}\label{gece}
\Gamma_\eps(c_\eps(t)) \equiv t \implies \gamma_\eps(c_\eps(t))c_\eps'(t) \equiv 1.
\end{equation}

In particular, one sees that if $u^\eps$ solves \eqref{DMNLSeps}, then
\[
w^\eps(t,x):=u^\eps(c_\eps(t),x)
\]
solves the equation
\begin{equation}\label{NLSeps}
i\partial_t w + \Delta w = c_\eps'(t)|w|^2 w.
\end{equation}

In the case of constant dispersion map $\gamma(t)\equiv\langle\gamma\rangle$, the same computations lead instead to the time-independent equation
\begin{equation}\label{NLS0}
i\partial_t w + \Delta w = \tfrac{1}{\langle\gamma\rangle} |w|^2 w.
\end{equation}

Note that the result of \cite{Dodson} implies that solutions to \eqref{NLS0} with $w|_{t=0}\in L^2$ are global and scatter in $L^2$, with $w\in S(\R)$.

We will prove the following homogenization result for \eqref{NLSeps}.

\begin{proposition}\label{P:w} Let $\varphi\in L^2$, and let $w\in S(\R)$ denote the solution to \eqref{NLS0} with $w|_{t=0}=\varphi$.  For all $\eps>0$ sufficiently small, the solution $w^\eps$ to \eqref{NLSeps} with $w^\eps|_{t=0}=\varphi$ is global in time, with $w^\eps\in S(\R)$ and
\[
\lim_{\eps\to 0}\|w^\eps-w\|_{S(\R)}=0. 
\]
\end{proposition}

\begin{proof} We follow the strategy in \cite{Ntekoume} and prove that $w$ is an approximate solution to \eqref{NLS0}. In particular, given $\eta>0$ small, we need to show that for all $\eps>0$ sufficiently small, we have
\begin{equation}\label{NTS}
\biggl\| \int_0^t e^{i(t-s)\Delta}[c_\eps'(s)-\tfrac{1}{\langle\gamma\rangle}]F(w(s))\,ds \biggr\|_{L_{t,x}^4(\R\times\R^2)} \lesssim \eta,
\end{equation}
where we denote $F(z)=|z|^2z$ and allow the implicit constants to depend on $\gamma$ and $\|w\|_{S(\R)}$.  We will show below that $c_\eps(s)-\tfrac{s}{\langle\gamma\rangle}=\mathcal{O}(\eps)$ uniformly in $s$, a fact we can exploit after an integration by parts in time. However, this produces derivatives acting on $F(w)$, which we can control only if we first restrict to low frequencies.  Thus, we will proceed by decomposing
\begin{equation}\label{decompose}
F(w) = [F(w)-F(w_{\leq N})] + F(w_{\leq N})= \text{O}(w_{>N}w^2) + F(w_{\leq N})
\end{equation}
for some $N\gg 1$. We will then show that the contribution of high frequencies can be made as small as we wish, and only use integration by parts on the low frequency piece.

We begin with the contribution of the $\O(w_{>N}w^2)$ term.  The key will be to prove the following:
\begin{equation}\label{high-freqs-small}
\|w_{>N}\|_{L_{t,x}^4(\R\times\R^2)} \lesssim \eta\qtq{for}N\qtq{sufficiently large.}
\end{equation}

To see this, first choose $N_0>0$ large enough that
\[
\|P_{>N_0}\varphi\|_{L^2} <\eta. 
\]

We then let $v$ denote the solution to \eqref{NLS0} with initial data $P_{\leq N_0}\varphi$.  In particular, by the scattering theory and persistence of regularity for \eqref{NLS0}, we have that $v$ is global in time and satisfies
\[
\|\langle\nabla\rangle v\|_{S(\R)}\lesssim N_0.
\]

Moreover, as
\[
\|w|_{t=0}-v|_{t=0}\|_{L^2}=\|P_{>N_0}\varphi\|_{L^2}<\eta, 
\]
the stability theory for \eqref{NLS0} implies that
\[
\|w-v\|_{L_{t,x}^4} \lesssim \eta
\]
provided $\eta$ is sufficiently small. Thus, by Bernstein's inequalities,
\begin{align*}
\|w_{>N}\|_{L_{t,x}^4} & \lesssim \|w-v\|_{L_{t,x}^4} + N^{-1}\|\nabla v\|_{L_{t,x}^4} \lesssim \eta + N^{-1}N_0,
\end{align*}
which yields \eqref{high-freqs-small} for $N\geq N_0\eta^{-1}$.

Using \eqref{high-freqs-small}, Strichartz, and \eqref{gece} (which implies $c_\eps' \in L_t^\infty$, with bound independent of $\eps$), we can now estimate
\begin{equation}\label{high-freq-part}
\begin{aligned}
\biggl\|  \int_0^t e^{i(t-s)\Delta}[c_\eps'(s)&-\tfrac{1}{\langle\gamma\rangle}]\text{\O}(w_{>N}w^2)(s)\,ds \biggr\|_{L_{t,x}^4(\R\times\R^2)} \\
& \lesssim\|c_\eps'-\tfrac{1}{\langle\gamma\rangle}\|_{L^\infty} \|w_{>N}\|_{L_{t,x}^4} \|w\|_{L_{t,x}^4}^2 \lesssim \eta,
\end{aligned}
\end{equation}
which is acceptable. 

We turn to the contribution of the $F(w_{\leq N})$ term in \eqref{decompose}.  We will prove 
\begin{equation}\label{low-freq-part}
\biggl\| \int_0^t e^{i(t-s)\Delta}\tfrac{d}{ds}[c_\eps(s)-\tfrac{s}{\langle \gamma\rangle}]\cdot F(w_{\leq N}(s))\,ds\biggr\|_{L_{t,x}^4(\R\times\R^2)}\lesssim N^{2}\eps.
\end{equation}
Integrating by parts with respect to time, we first write
\begin{align}
\int_0^t & e^{i(t-s)\Delta}\tfrac{d}{ds}[c_\eps(s)-\tfrac{s}{\langle \gamma\rangle}] \cdot F(w_{\leq N}(s))\,ds \nonumber \\
& = [c_\eps(t)-\tfrac{t}{\langle\gamma\rangle}] F(w_{\leq N}(t)) \label{boundary-term} \\
& \quad -i\int_0^t e^{i(t-s)\Delta}[c_\eps(s)-\tfrac{s}{\langle\gamma\rangle}](i\partial_s + \Delta)F(w_{\leq N}(s))\,ds \label{derivative-term}, 
\end{align}
where we have used the fact that $c_\eps(0)=0$.  Thus it suffices to estimate the contribution of \eqref{boundary-term} and \eqref{derivative-term}.  For each of these terms, the key will be to use the following bound:
\begin{equation}\label{eps-gain}
\|c_\eps(t)-\tfrac{t}{\langle\gamma\rangle}\|_{L_t^\infty} \lesssim_\gamma \eps,
\end{equation}
which we may prove as follows: by Lemma~\ref{L:crucial-estimate} and the definition of $c_\eps$, we have
\begin{align*}
|c_\eps(\Gamma_\eps(t)) - \tfrac{1}{\langle\gamma\rangle}\Gamma_\eps(t)| & = \tfrac{1}{\langle\gamma\rangle}|\langle \gamma\rangle t - \Gamma_\eps(t)| \lesssim_\gamma \eps
\end{align*}
uniformly in $t\in\R$. As $\Gamma:\R\to\R$ is surjective, the estimate \eqref{eps-gain} follows.

We turn to the estimate of \eqref{boundary-term} and \eqref{derivative-term}.  For \eqref{boundary-term}, we use Bernstein's inequality (with the fact that $F(w_{\leq N})$ is still localized to frequencies $\lesssim N$) to estimate
\begin{align*}
\|[c_\eps(t)-\tfrac{t}{\langle\gamma\rangle}] F(w_{\leq N})\|_{L_{t,x}^4} & \lesssim \|c_\eps(t)-\tfrac{t}{\langle\gamma\rangle}\|_{L_t^\infty} \| F(w_{\leq N})\|_{L_{t,x}^4} \\
& \lesssim \eps N^{\frac32} \| F(w_{\leq N})\|_{L_t^4 L_x^1} \\ &\lesssim \eps N^{\frac32} \|w\|_{L_t^\infty L_x^2}\|w_{\leq N}\|_{L_t^\infty L_x^4} \|w\|_{L_{t,x}^4} \\
& \lesssim  \eps N^{2}\|w\|_{L_t^\infty L_x^2}^2 \|w\|_{L_{t,x}^4} \lesssim \eps N^{2},
\end{align*}
which is acceptable.

For \eqref{derivative-term}, we apply Strichartz and \eqref{eps-gain} and find that it will suffice to prove
\[
\|\partial_s F(w_{\leq N})\|_{L_{t,x}^\frac{4}{3}} + \| \Delta F(w_{\leq N})\|_{L_{t,x}^{\frac{4}{3}}} \lesssim N^2. 
\]

For the second term, we use the frequency localization and H\"older's inequality to obtain
\[
\|\Delta F(w_{\leq N})\|_{L_{t,x}^{\frac43}} \lesssim N^2\|w\|_{L_{t,x}^4}^3 \lesssim N^2,
\]
which is acceptable.  For the first term, we use the fact that $w_{\leq N}$ solves
\[
i\partial_t w_{\leq N} = - \Delta w_{\leq N} + \tfrac{1}{\langle \gamma\rangle} P_{\leq N}(|w|^2 w)
\]
and Bernstein and H\"older to estimate as follows:
\begin{align*}
\|\partial_s F(w_{\leq N})\|_{L_{t,x}^{\frac43}} & \lesssim \| w_{\leq N}^2 \Delta w_{\leq N}\|_{L_{t,x}^{\frac43}} + \|w_{\leq N}^2 P_{\leq N}(|w|^2 w)\|_{L_{t,x}^{\frac43}} \\
& \lesssim N^2\|w\|_{L_{t,x}^4}^3 + \|w_{\leq N}\|_{L_{t,x}^\infty}^2\|P_{\leq N}(|w|^2 w)\|_{L_{t,x}^{\frac43}} \\
& \lesssim N^2+ N^2\|w\|_{L_t^\infty L_x^2}^2\|w\|_{L_{t,x}^4}^3 \lesssim N^2,
\end{align*}
which is acceptable.  This completes the proof of \eqref{low-freq-part}.

Combining \eqref{high-freq-part} and \eqref{low-freq-part}, we obtain
\[
\biggl\| \int_0^t e^{i(t-s)\Delta}[c_\eps'(s)-\tfrac{1}{\langle\gamma\rangle}] F(w(s))\,ds \biggr\|_{L_{t,x}^4} \lesssim \eta + N^2\eps.
\]
Choosing $\eps$ sufficiently small, we therefore obtain \eqref{NTS}. 

With \eqref{NTS} in place, we can now apply the stability result (Proposition~\ref{L:stability}) to obtain that
\[
\|w^\eps-w\|_{S(\R)} \lesssim \eta
\]
for all $\eps$ sufficiently small.  As $\eta>0$ was arbitrary, this completes the proof of Proposition~\ref{P:w}. \end{proof}

It remains to show that Proposition~\ref{P:w} implies the desired averaging result appearing in Theorem~\ref{T}(ii).

\begin{proof}[Proof of Theorem~\ref{T}(ii)] We define $u$ and $u^\eps$ as in the statement of Theorem~\ref{T}(ii), and define the corresponding solutions $w$ and $w^\eps$ to \eqref{NLS0} and \eqref{NLSeps} as in Proposition~\ref{P:w}.  In particular, Proposition~\ref{P:w} guarantees that the $u^\eps$ are global in time for small enough $\eps>0$ and satisfy
\begin{equation}\label{converge0}
\lim_{\eps\to 0} \|u^\eps(c_\eps(t),\cdot) - u(\tfrac{t}{\langle\gamma\rangle},\cdot) \|_{S(\R)} = 0.
\end{equation}

We will prove that 
\begin{equation}\label{converge1}
\lim_{\eps\to 0}\|u(c_\eps(t),\cdot)-u(\tfrac{t}{\langle\gamma\rangle},\cdot)\|_{S(\R)} = 0,
\end{equation}
so that by the triangle inequality we obtain 
\[
\lim_{\eps\to 0} \|u^\eps(c_\eps(t),\cdot)-u(c_\eps(t),\cdot)\|_{S(\R)}=0.
\]

As $c_\eps:\R\to\R$ is invertible and \eqref{gece} guarantees that $\tfrac{1}{c_\eps'}\in L_t^\infty$ (uniformly in $\eps$), a change of variables then implies the desired convergence \eqref{convergence!}.

It therefore remains to establish \eqref{converge1}.  We let $\eta>0$ and consider the $L_t^\infty L_x^2$- and $L_{t,x}^4$-norms separately.

First, for the $L_{t,x}^4$-norm, we choose $\psi\in C_c^\infty(\R\times\R^2)$ so that
\[
\|u-\psi\|_{L_{t,x}^4(\R\times\R^2)}<\eta.
\]

Then, using a change of variables, the fundamental theorem of calculus, and \eqref{eps-gain}, we have
\begin{align*}
\|u(c_\eps(t))-u(\tfrac{t}{\langle \gamma\rangle})\|_{L_{t,x}^4} & \lesssim \|\psi(c_\eps(t))-\psi(\tfrac{t}{\langle\gamma\rangle})\|_{L_{t,x}^4}+\eta \\
& \lesssim \|c_\eps(t)-\tfrac{t}{\langle\gamma\rangle}\|_{L_t^\infty} \|\partial_t\psi\|_{L_{t,x}^\infty}\|1\|_{L_{t,x}^4(\text{supp}(\psi))}+\eta \\
& \lesssim C(\psi)\eps + \eta.
\end{align*}

Next, for the $L_t^\infty L_x^2$-norm, we first argue as we did for \eqref{high-freqs-small} above to find $N=N(\eta)$ sufficiently large that
\[
\|u_{>N}\|_{L_{t,x}^4}\lesssim \eta. 
\]

Writing $I_\eps(t)$ for the interval between $c_\eps(t)$ and $\tfrac{t}{\langle \gamma\rangle}$, we use the Duhamel formula, \eqref{decompose}, Strichartz estimates, Bernstein estimates, and \eqref{eps-gain} to obtain 
\begin{align*}
\|u(c_\eps(t))-u(\tfrac{t}{\langle\gamma\rangle})\|_{L_x^2} & = \tfrac{1}{\langle\gamma\rangle}\biggl\| \int_{I_\eps(t)} e^{-is\Delta}[\text{\O}(u_{>N}u^2)+F(u_{\leq N})]\,ds\biggr\|_{L_x^2} \\
& \lesssim \|u_{>N}\|_{L_{t,x}^4}\|u\|_{L_{t,x}^4}^2 + \|F(u_{\leq N})\|_{L_{t,x}^{\frac43}(I_\eps(t))} \\
& \lesssim \eta + |I_\eps(t)|^{\frac34} \|u_{\leq N}\|_{L_t^\infty L_x^4}^3 \\
& \lesssim \eta + \eps^{\frac34} N^{\frac32}\|u\|_{L_t^\infty L_x^2}^3 \lesssim \eta + \eps^{\frac34}N^{\frac32} 
\end{align*}
uniformly in $t$.

Combining the estimates above, we find
\[
\|u(c_\eps(t),\cdot)-u(\tfrac{t}{\langle\gamma\rangle},\cdot)\|_{S(\R)} \lesssim \eta + C(\psi)\eps + \eps^{\frac34}N^{\frac32}.
\]

Choosing $\eps$ sufficiently small, we deduce 
\[
\|u(c_\eps(t),\cdot)-u(\tfrac{t}{\langle\gamma\rangle},\cdot)\|_{S(\R)} \lesssim \eta.
\]

As $\eta>0$ was arbitrary, we obtain \eqref{converge0}, as desired. \end{proof}


\end{document}